\documentclass{article}
\usepackage{amsmath}
\usepackage{ntheorem}
\usepackage{amssymb}
\usepackage{authblk} 
\usepackage{algorithm}
\usepackage{algorithmicx}
\usepackage{algpseudocode}
\usepackage{graphicx}
\usepackage{subfigure}
\usepackage{epstopdf}
\usepackage{tabularx}  
\usepackage{booktabs}
\usepackage[numbers, sort]{natbib}
\usepackage{hyperref}
\usepackage[top=2cm, bottom=2cm, left=2cm, right=2cm]{geometry}  
\usepackage{xcolor} 
\usepackage[export]{adjustbox}
\usepackage{color}
\usepackage[title]{appendix}
\usepackage[T1,T2A]{fontenc}
\usepackage[utf8]{inputenc}
\usepackage[russian, german, english]{babel}

\usepackage{authblk}

\hypersetup{urlcolor=blue, citecolor=red}

\newtheorem{theorem}{Theorem}

\newtheorem*{proof}{Proof}

\newtheorem{lemma}{Lemma}

\def\RR{\mathbb{R}}

\title{On the curvature bounded sphere problem in $\RR^3$}
\author[1]{Hongda Qiu}
\affil[1]{Email: hjq5042@psu.edu, Department of Mathematics, The Pennsylvania State University}

\begin{document}
\maketitle

\begin{abstract}
	We prove that if a topological sphere smoothly embedded into $\RR^3$ with normal curvatures absolutely bounded by $1$ is contained in an open ball of radius $2$, then the region it bounds must contain a unit ball. This result suggests a potential direction for a novel problem by Dmitri Burago and Anton Petrunin, which asks whether a topological sphere smoothly embedded in $\RR^3$, with all normal curvatures absolutely bounded by $1$, necessarily encloses a volume of at least $\frac{4}{3}\pi$. For completeness, the appendix presents an example illustrating an alternative aspect on this problem. 
\end{abstract}
\section{Introduction}
Volume estimates illustrate how the intrinsic curvature of a manifold constrains its embedding or immersion into an ambient space. These estimates are typically studied in broad generality (see, for instance, \cite[Chapter~7.1]{petersen2006riemannian}, \cite[Sections~34-35]{burago2013geometricinequalities}). As the exploition deepens, however, some overlooked topics arise to be interesting, even in the context of simple geometric objects.

In the early 21th century, Dmitri Burago and Anton Petrunin \cite{anton2022mathoverflow} formulated a novel problem: Let $S$ be a topological sphere smoothly embedded in $\RR^3$ with normal curvatures absolutely bounded by $1$. Is the volume enclosed by $S$ at least that of a unit ball, namely $\frac{4}{3}\pi$?\footnote{The original problem was considered in a three-dimensional non-positively curved space. It was later realized that the Euclidean case implies the Hadamard case.} This problem has inspired two directions of discussion---one aiming to prove a positive answer, and the other seeking a counterexample.

Regarding the positive direction, one might intuitively attempt to generalize the Pestov--Ionin theorem \cite{pestov1959}\footnote{The original reference is in Russian. For an English reference, see, for instance, \cite[Section E, Chapter 7, p.~84]{petrunin2024differentialgeometrycurvessurfaces}} that a plane curve with curvature bounded by $1$ contains a unit disc, to its three-dimensional analogue. This idea holds when the body enclosed by the surface is convex, but it fails in general: there exists closed surfaces with normal curvatures absolutely bounded by $1$ that enclose volume $\geq\frac{4}{3}\pi$ but do not contain a unit ball \cite[p.~11]{petrunin2023pigtikalpuzzlesgeometryi}. Moreover, it is known that the maximal radius of a ball that can be contained in any smooth topological sphere with normal curvatures absolutely bounded by $1$ is strictly less than $1$ \cite[Section 30.4, p.~231]{burago2013geometricinequalities} \cite{lagunov1963extremalproblemssurfacesprescribed}.\footnote{This is an English translation by Richard Bishop of the original Russian work by Lagunov and Fet \cite{lagunov1960russianI}, though some steps are only sketched.}

A question related to the positive direction is whether the unit sphere is locally optimal in the sense that any small enough smooth perturbation increases the volume it encloses. In this work, we prove a slightly stronger statement:
\begin{theorem}
	\label{theorem enclosed unit ball}
	If the surface $S$, as defined above, is additionally contained by an open ball of radius $2$, then the region it bounds must contain a unit ball.
\end{theorem}

The proof of Theorem \ref{theorem enclosed unit ball} is presented in the next Section. It relies on the observation that the body bounded by $S$ is \textit{star-shaped}; that is, every point on $S$ can be joined to $O$ by a line segment and this line segment lies in the open region bounded by $S$ \cite[Sections~35.1, p.~256]{burago2013geometricinequalities}. Consequently, our method fails when the radius of the open ball is increased to $2+\delta$ for any positive $\delta$ (for instance, consider $S$ as the unit sphere centered at a point located $1+\frac{\delta}{2}$ from $O$).

This statement offers some insight into the positive direction for the Burago--Petrunin problem. A potential approach to a positive resolution to the Burago--Petrunin problem is to divide it into two cases: one where $S$ is contained in some open ball of finite radius and another where it is not. These cases can then be treated separately. One might improve our result by increasing the radius of the bounding sphere in the statement, though this needs additional techniques as explained above. For the latter case, a potential first step is to prove the result when the radius is large.

As a remark, however, the statemnt does not yield any affirmative conclusion in either the positive or negative direction towards the Burago--Petrunin problem. There is still a hope to find a counterexample, though its construction might be highly nontrivial. For instance, one might construct a counterexample by modifying Lagunov's Fishbowl \cite{lagunov1961russianII}.\footnote{For an English reference, see, for instance, \cite[Section C, Chapter 11, p.~115]{petrunin2024differentialgeometrycurvessurfaces}} For completeness, an example suggested by Anton Petrunin \cite{anton2022mathoverflow} is included in Section \ref{appendix fishbowl} as an appendix---it is a surface of genus $2$ enclosing a volume of $ \frac{22}{3}\pi - 2\pi^2 +\epsilon \approx 3.3 +\epsilon $, where the positive constant $\epsilon$ can be arbitrarily small. Here the value $3.3$ is less than the volume of a unit ball, $\frac{4}{3}\pi\approx 4.2$. Moreover, the Burago-Petrunin problem might also be approached from alternative perspectives, such as by finding a sufficient and necessary condition for $S$ to enclose a volume of $ \frac{4}{3}\pi $. One might ask the same question for surfaces of genus $1$, $3$, and higher.

\textbf{Acknowledgement} I would like to thank Dmitri Burago and Anton Petrunin for their help and suggestions on this work. 

\textbf{Funding Statement} There are no funding statement for this work.

\section{Proof of Theorem \ref{theorem enclosed unit ball}}
\label{section proof}
Let $\Omega$ be a sphere of radius $2$ centered at the origin $O$ and assume that $S$ is contained in the open ball bounded by $\Omega$.

We begin by showing that the body bounded by $S$ is star-shaped. To achieve this goal, we consider the map projecting every point on $\Omega$ radially onto $S$ and show that it is both invertible and \textit{short}---that is, it does not increase length. 
\begin{lemma}
	\label{lemma radial projection}
	The radial projection, denoted by $\rho:\Omega\to S$, is a diffeomorphism. Moreover, it is short.
\end{lemma}
The proof of Lemma~\ref{lemma radial projection} relies on the following lemma. 
\begin{lemma}
	\label{lemma inequality}
	Let $S$ be the surface as defined above and $x$ any point on $S$. Denote by $\alpha = \alpha(x)$ the angle between the vector $\overrightarrow{Ox}$ and the normal vector to $S$ at $x$. Then there exists a positive constant $\epsilon$ such that
	\begin{equation}
		\label{equation inequality}
		\epsilon\leq|\overrightarrow{Ox}|\leq2\cos{\alpha}.
	\end{equation}
	Moreover, the origin $O$ lies in the open region enclosed by $S$.
\end{lemma}
The second part of Inequality \ref{equation inequality} is proved by Petrunin \cite[Lemma 4.1]{petrunin2024gromovtori}. We prove the first part and the final assertion.
\begin{proof}
	We prove the inequality by contradiction. Suppose that $S$ passes through the origin $O$. Note that the curvatures of geodesics on $S$ are bounded above by $1$. Consider a geodesic on $S$ of length $\pi$ starting from $O$ (such a geodesic exists because closed geodesics on $S$ have length at least $2\pi$).  By the Bow Lemma \footnote{The original version is proved by Schur \cite{schur1921bowlemma} and generalized by Schmidt \cite{schmidt1925bowlemma}. For an English reference, see \cite[Theorem~3.19, p.~51]{petrunin2024differentialgeometrycurvessurfaces}.}, the distance between the two ends of this geodesic is no less than that of a unit semicircular arc, which is $2$. It follows that this geodesic intersects $\Omega$---a contradiction, and the inequality is proved.
	
	Now we turn to the final assertion of the lemma. We proceed by contradiction. Suppose that the region bounded by $S$ does not contain the origin $O$. Then there exists a line segment starting from $O$ and tangent to $S$ at some point $x_0\in S$. It follows that $\cos\alpha(x_0) = 0 <|x_0|$, a contradiction.
\end{proof}

\begin{proof}[Proof of Lemma~\ref{lemma radial projection}]
	Using Lemma~\ref{lemma inequality}, we can choose a unit normal vector field $\nu$ on $S$ such that $\epsilon\leq|x|\leq2\cos{\angle(x,\nu(x))}$ for every $x\in S$. Here $\epsilon$ refers to the positive constant in Lemma~\ref{lemma inequality}.
	
	Consider the map $ S\to\Omega $ defined by $x\mapsto\frac{2x}{|x|}$. The inequality $ \cos{\angle(x,\nu(x))} \geq \epsilon \geq 0 $ implies that the angle between $x$ and $\nu(x)$ is uniformly bounded from $\frac{\pi}{2}$, so this map is locally invertible. Since the region bounded by $S$ contains the origin $O$, the map is globally invertible. Its inverse defines the radial projection $\rho:\Omega\to S$.
	
	Next, we show that $\rho$ is short. Let us denote $ p = \frac{2x}{|x|} $. It follows that for any tangent vector $v$ to $S$ at $p$, the ratio between $|x|$ and $|p|$ is no less than that between the projection of the differential $D_p\rho(v)$ onto the direction of $|p|$---that is, $ D_p\rho(v)\cos{\angle(\nu(x),x)} $---and $|v|$. That is\footnote{The equality holds when $v$ lies in the plane spanned by $x$ and $\nu(x)$, as illustrated in Figure~\ref{figure short map}.},
	\begin{equation}
		\label{inequality short}
		\frac{D_p\rho(v)\cos{\angle(\nu(x),x)}}{|v|} \leq \frac{|x|}{|p|} = \frac{|x|}{2}.
	\end{equation}
	\begin{figure}[htbp]
		\centering
		\caption{An illustration for Inequality \ref{inequality short} when $v$ lies in the plane spanned by $x$ and $\nu(x)$.}
		\label{figure short map}
		\includegraphics[width=1\linewidth, cframe=black 0.1mm]{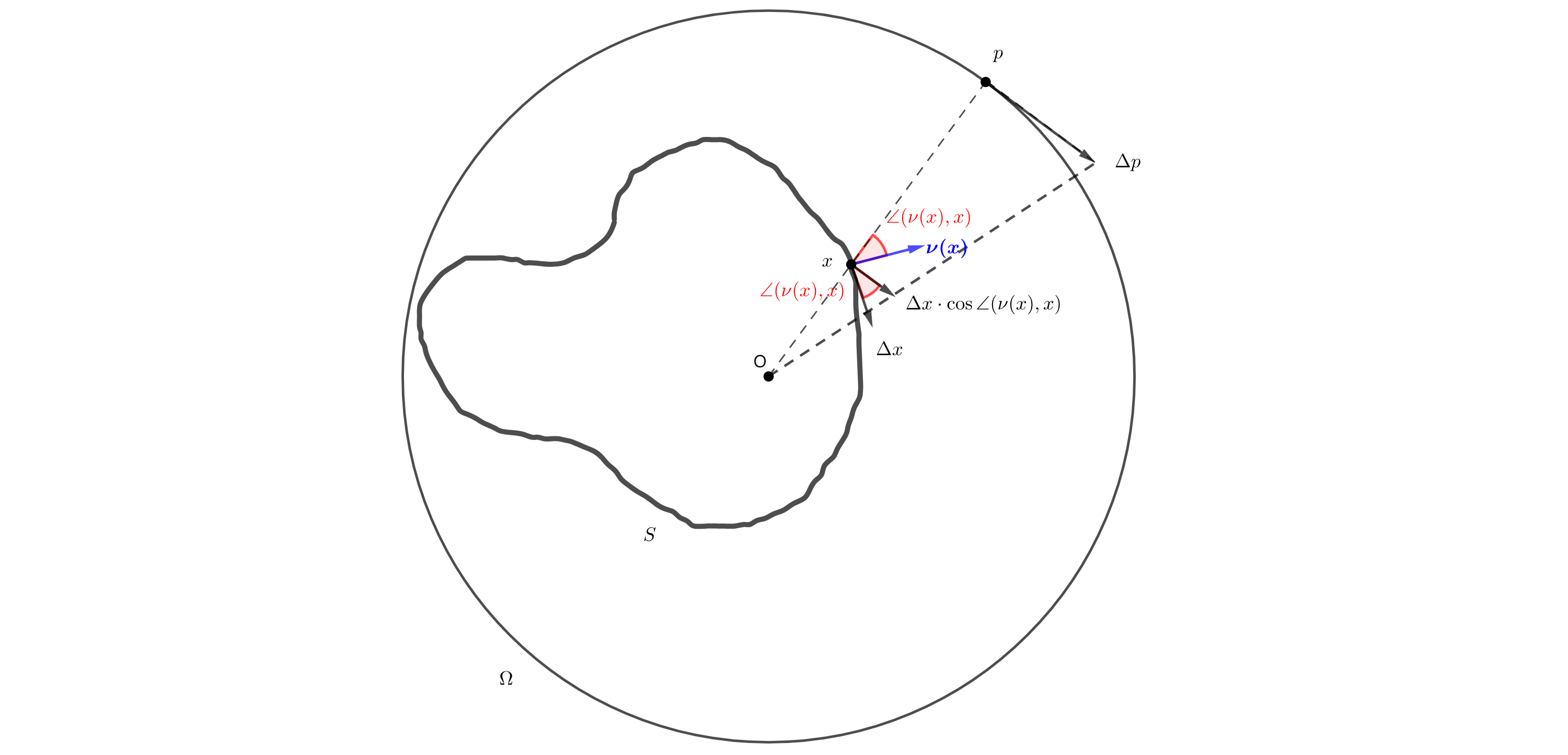}
	\end{figure}
	Recalling Inequality~\ref{equation inequality}, we have $|D_p\rho(v)|\leq|v|$, so $\rho$ is short.
\end{proof}
To find a unit ball enclosed by $S$, we need the following lemma.
\begin{lemma}
	\label{lemma enclosed ball}
	If $x\in S$ attains the maximal distance from $S$ to the origin $O$, then the ball having the line segment $\overline{Ox}$ as one of its diameters is contained in the region bounded by $S$.
\end{lemma}
\begin{proof}[Proof of Lemma~\ref{lemma enclosed ball}]
	Denote the ball in consideration as $B_1$.
	
	We prove this lemma by contradiction. Consider the point $y\in S$ closest to the center of $B_1$, namely $\frac{x}{2}$. Let $x=\rho(p)$ and $y=\rho(q)$. Assume that $y$ lies in the region bounded by $S$; that is, $|\frac{x}{2}-y|< |\frac{x}{2}|$ (Figure \ref{figure enclosed ball}).
	\begin{figure}[htbp]
		\centering
		\caption{An illustration for the assumption used in the proof of Lemma \ref{lemma enclosed ball}.}
		\label{figure enclosed ball}
		\includegraphics[width=1\linewidth,cframe=black 0.1mm]{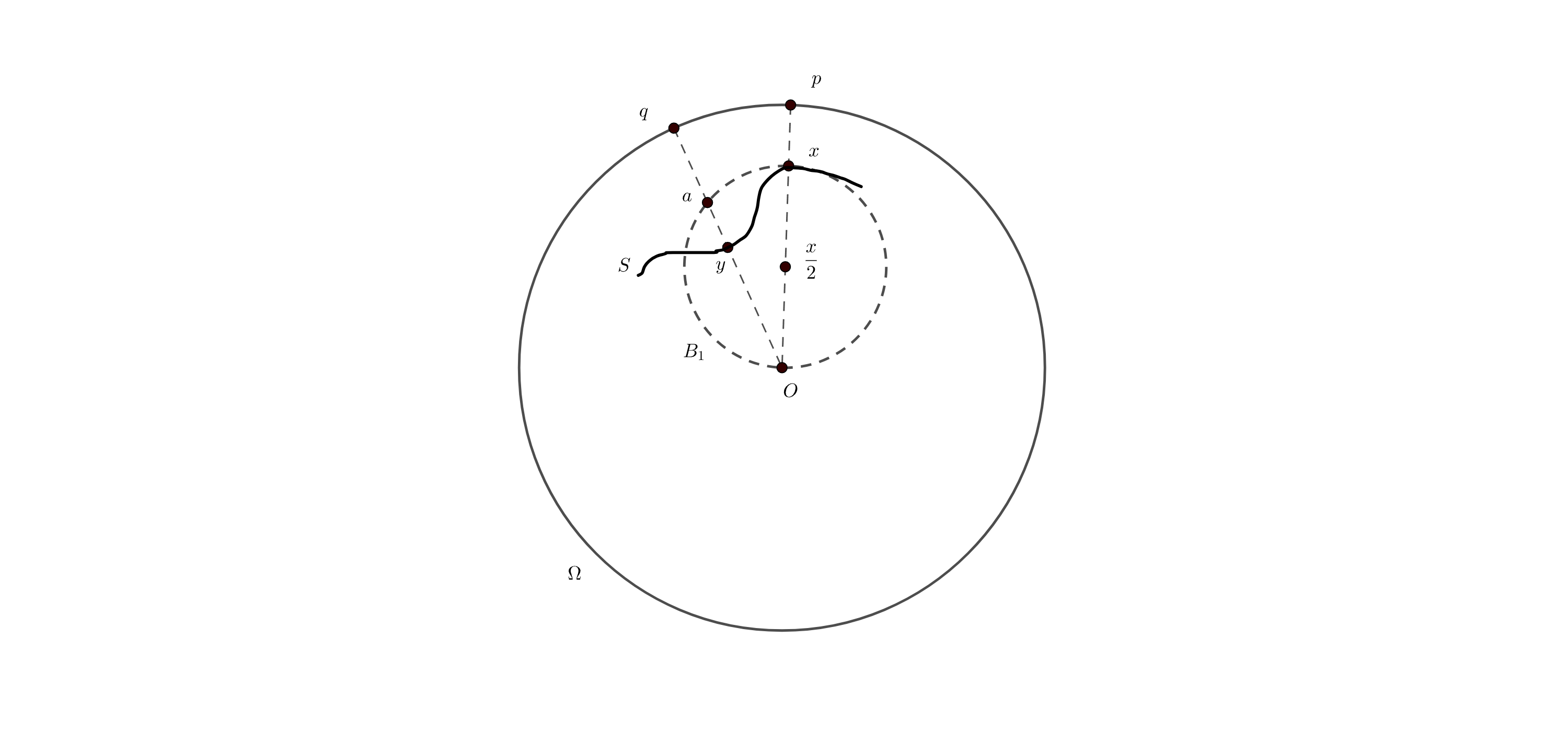}
	\end{figure}
	
	Denote by $a$ the intersecting point of the line segment $\overline{Oq}$ and the ball $B_1$. We get that
	\begin{align*}
		&\angle(\nu(x),\nu(y))& \\
		=& \angle(x,y-\frac{x}{2})\tag{since $y$ is closest to $\frac{x}{2}$}\\
		>&\angle(a\frac{x}{2}p)\tag{since $y$ is a point on the chord $\overline{Oa}$ of $B_1$}\\\
		=&2\angle(pOq)\tag{since $O,x,a$ lie on the same circle}\\
		=&d_\Omega(p,q)\\
		\geq&d_S(x,y)\tag{by Lemma~\ref{lemma radial projection}}.
	\end{align*}
	Consider the shortest geodesic on $S$ between $x$ and $y$. Since any geodesic on $S$ has curvature at most $1$, the integral of the curvature of this geodesic is bounded above by its length. It follows that the last term $d_S(x,y)$ is no less than $ \angle(\nu(x),\nu(y)) $, which is a contradiction.
\end{proof}
\begin{proof}[Proof of Theorem \ref{theorem enclosed unit ball}]
	To finish the proof, it suffices to translate the surface $S$ so that the point $x$ that maximizes the distance from $S$ to the origin $O$ gets arbitrarily close to the sphere $\Omega$ and apply Lemma \ref{lemma enclosed ball}. It follows that the unit ball tangent to $S$ at $x$ is enclosed by $S$.
\end{proof}

\section{Appendix: Petrunin's Fishbowl}
\label{appendix fishbowl}
This section is included only for completeness. We construct a closed surface of genus $2$ with all normal curvatures absolutely bounded by $1$, and whose enclosed volume is less than that of a unit ball, as follows.

We begin with two surfaces of revolution $S_1, S_2$, each of which has a cross section consisted of two horizontal lines and two half circles of radius no less than $1$ as shown in Figure~\ref{figure two surfaces}. We assume that these horizontal lines are long enough (to the extend that will be clear in a minute) and the space between them is thin. We then attach three tunnels onto $S_1,S_2$ as follows: one (the green one) joining the surfaces together, and the other two (the red one and the blue one) forming handle-like attachments on each of $S_1$ and $S_2$, respectively. The green tunnel has radius $1$, except near its ends, which are constructed by rovolving a quarter arc of the unit circle. The red and blue tunnels have radius slightly larger than $1$ (think of $1+\delta$ with a small positive constant $\delta$), and their ends are constructed akin to those of the green one. We let the horizontal lines of the cross sections of $S_1$ and $S_2$ be long enough so that the normal curvatures of the red and green tunnels remains no more than $1$ on the large upper curved parts.

It follows that most of the volume of the surface concentrates in a body of revolution enclosed by one red joint, one blue joint and the green tunnel (which appears in the right middle part of Figure \ref{figure fishbowl}). One may think of this body as the revolution of the plane area enclosed by the line $ x=1 $ and the unit circles centered at $(2,1),(2,-1)$ around the $y$-axis. This body has a volume of $\frac{22}{3}\pi - 2\pi^2\approx3.3$. The volume of all remaining thin part of the surface can be arbitrarily small.

Topologically, this surface consists of two spheres and three tunnels (Figure \ref{figure topology graph}), which has genus $2$. We remark that one might be able to modify it---specifically, by gluing four discs---to obtain a topological sphere without increasing its enclosed volume too much so as to construct a counterexample for the Burago-Petrunin problem.
\begin{figure}[htbp]
	\centering
	\begin{subfigure}
		\centering
		\caption{The surfaces $S_1,S_2$}
		\label{figure two surfaces}
		\includegraphics[width=0.5\linewidth,cframe=black 0.1mm]{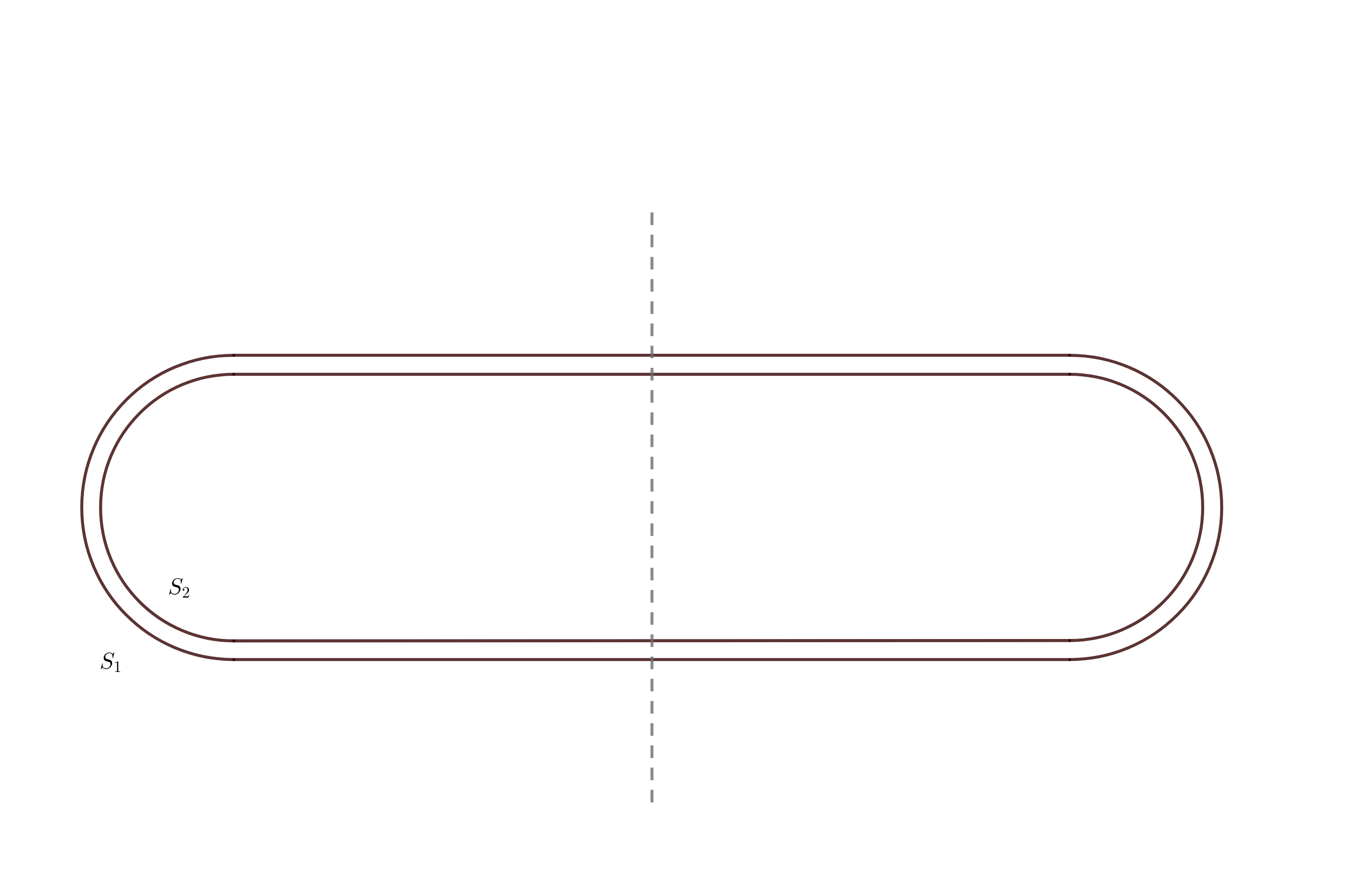}
	\end{subfigure}
	\begin{subfigure}
		\centering
		\caption{Petrunin's Fishbowl}
		\label{figure fishbowl}
		\includegraphics[width=0.5\linewidth,cframe=black 0.1mm]{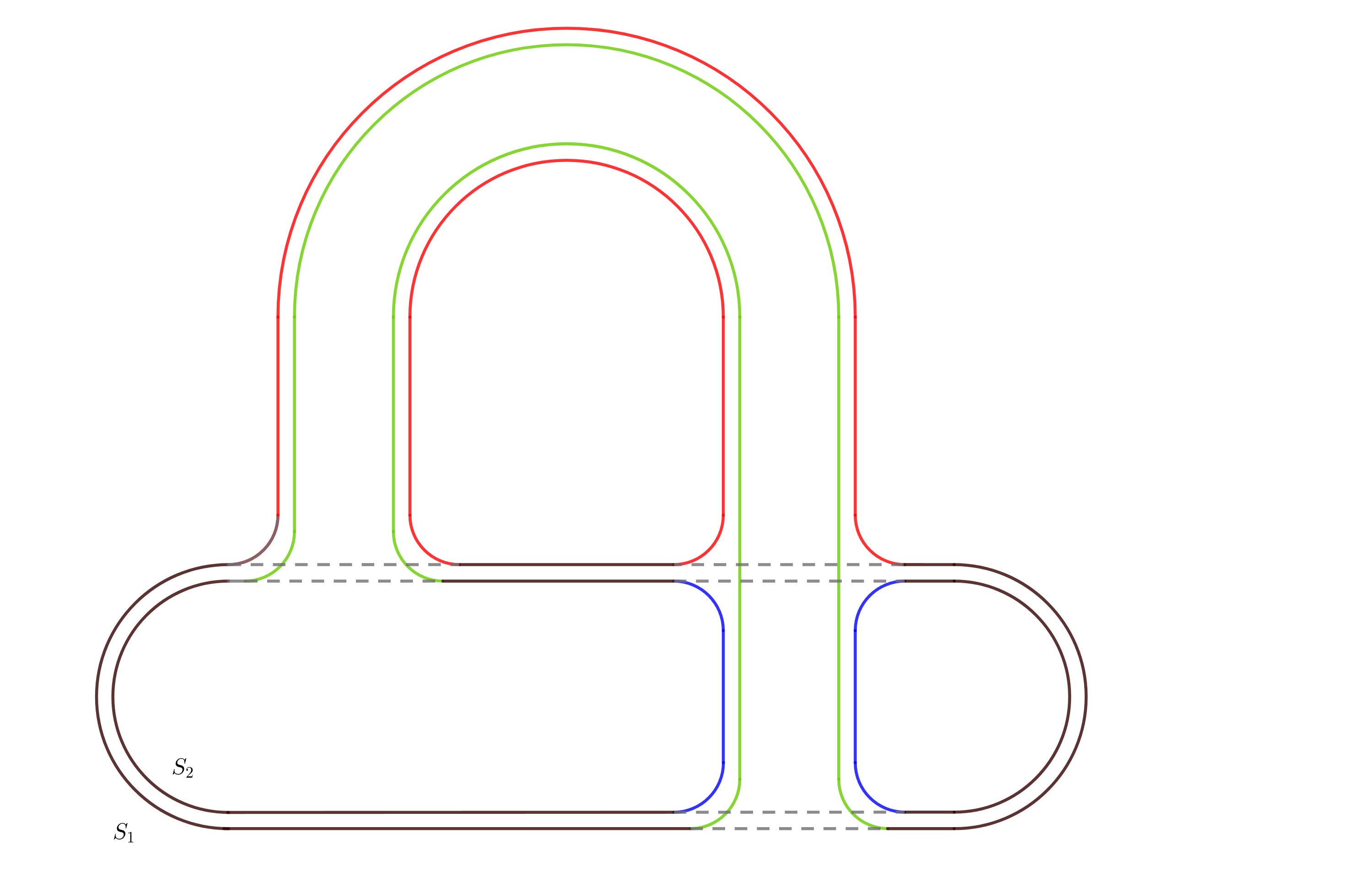}
	\end{subfigure}
	\begin{subfigure}
		\centering
		\caption{The ``topological equivalence" of Petrunin's Fishbowl}
		\label{figure topology graph}
		\includegraphics[width=0.5\linewidth,cframe=black 0.1mm]{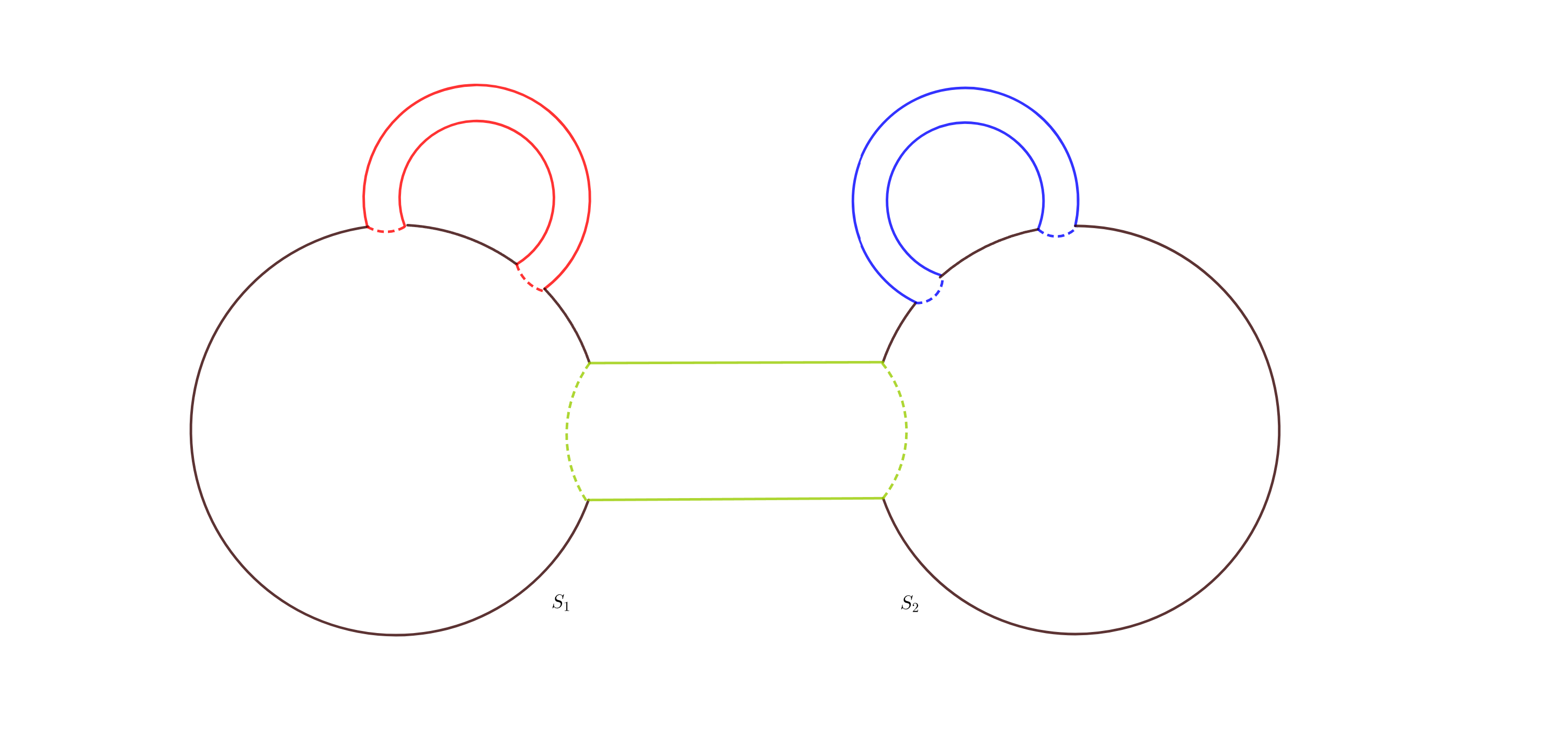}
	\end{subfigure}
\end{figure}

\newpage
\bibliographystyle{unsrt}
\bibliography{ref}
\end{document}